\documentclass{amsart}
	
\usepackage{amsmath,amssymb,amsthm,amsfonts,enumerate,url,fancyhdr,mathbbol,mathrsfs,color}
\usepackage[sort,numbers]{natbib}

\usepackage[utf8]{inputenc}
\def\me{\mathsf{e}}
\def\mv{\mathsf{v}}
\def\:{\thinspace:\thinspace}
\newtheorem{theo}{Theorem}
\newtheorem{lemma}[theo]{Lemma}

\newtheorem{cor}[theo]{Corollary}

\newtheorem{defi}[theo]{Definition}

\newtheorem{rem}[theo]{Remark}

\numberwithin{equation}{section}
\numberwithin{theo}{section}

\theoremstyle{definition}

\numberwithin{theo}{section}

\newcommand{\R}{{\mathbb R}}
\newcommand{\N}{{\mathbb N}}
\newcommand{\Z}{{\mathbb Z}}

\DeclareMathOperator{\dist}{dist}

\DeclareMathOperator{\grad}{grad}
\DeclareMathOperator{\diver}{div}

  \def\mG{\mathsf{G}}
  \def\mV{\mathsf{V}}
  \def\mE{\mathsf{E}}
  
  \def\Lfun{\mathscr{L}}

\def\mv{\mathsf{v}}
 \def\me{\mathsf{e}}
 \def\mw{\mathsf{w}}
  \def\mW{\mathsf{W}}
  \def\mf{\mathsf{f}}

\title[Regularity and long-time behavior of $p$-Laplace equations]{Time regularity and long-time behavior of parabolic $p$-Laplace equations on infinite graphs}
 \author{Bobo Hua}
\address{Bobo Hua, School of Mathematical Sciences, LMNS, Fudan University, 200433 Shanghai, China}
\email{bobohua@fudan.edu.cn}

\author{Delio Mugnolo}
\address{Delio Mugnolo, Lehrgebiet Analysis, FernUniversität in Hagen, 58084 Hagen, Germany}
\email{delio.mugnolo@fernuni-hagen.de}


\thanks{This research has been partially supported by the Land Baden--Württemberg in the framework of the \emph{Juniorprofessorenprogramm} -- research project on ``Symmetry methods in quantum graphs'' -- and by the Center for Interdisciplinary Research (ZiF) in Bielefeld in the framework of the cooperation group on ``Discrete and continuous models in the theory of networks".}

\begin{document}

\begin{abstract}
We consider the so-called \emph{discrete $p$-Laplacian}, a nonlinear difference operator that acts on functions defined on the nodes of a possibly infinite graph. We study the associated nonlinear Cauchy problem and identify the generator of the associated nonlinear semigroups. We prove higher order time regularity of the solutions. We investigate the long-time behavior of the solutions and discuss in particular finite extinction time and conservation of mass. Namely, on one hand, for small $p$ if an infinite graph satisfies some isoperimetric inequality, then the solution to the parabolic $p$-Laplace equation vanishes in finite time; on the other hand, for large $p,$ these parabolic $p$-Laplace equations always enjoy conservation of mass.
\end{abstract}

\maketitle

\section{Introduction}\label{sec:1}

The $p$-Laplacian
\begin{equation}\label{form:plapl}
u\mapsto \Delta_p u:=-\mathcal \diver \left(|\grad u|^{p-2}\mathcal \grad u\right),\qquad p\in [1,\infty)\ ,
\end{equation}
is a well-known operator in the theory of nonlinear partial differential equations. It is commonly applied to functions $u$ supported on domains of $\mathbb R^d$ but following~\cite{Tak91} in recent years it has been increasingly investigated in the context of spectral geometry on Riemannian manifolds, too.

Throughout this paper we consider instead finite or countable, simple, connected, oriented, weighted graphs $\mG=(\mV,\mE,\nu,\mu)$, where $\mu(\me),\nu(\mv)$ denote the weights assigned to the edge $\me\in \mE$ and to the node $\mv\in \mV$, respectively. Indeed, it is also possible to define a discrete version of the $p$-Laplacian that acts on functions which are defined on the node set $\mV$ of $\mG$: Following the standard ansatz of finite difference calculus, a \emph{discrete $p$-Laplacian} then arises as the divergence operator is replaced by the incidence matrix $\mathcal I$ of the graph, and accordingly the gradient by the transpose of $\mathcal I$. This leads to
\begin{equation}\label{eq:defin}
f\mapsto \Lfun_p f:=\mathcal I \left(|\mathcal I^T f|^{p-2}\mathcal I^T f\right)\ ,
\end{equation}
or rather
\begin{equation}\label{eq:defin-2}
f\mapsto \Lfun_p f:=\frac{1}{\nu}\mathcal I \Big(\mu |\mathcal I^T f|^{p-2}\mathcal I^T f\Big)\ ,
\end{equation}
in the case of nontrivial weights $\mu,\nu$. These expressions are well-defined whenever $\mG$ is a finite graph but need to be refined for infinite ones, cf.\ Section~\ref{sec:2}.  We stress that the operator $\Lfun_p$ introduced by means of the oriented incidence matrix $\mathcal I$ is actually independent of the orientation.

The discrete $p$-Laplacian was introduced in~\cite{NakYam76} and has been well studied ever since, mostly in the context of nonlinear potential and spectral theory, cf.~\cite{BerJebSer13,IanTer13,Mug13} for an historical overview. Its parabolic properties have been investigated less extensively. In the case of graphs on a finite node set $\mV$ the \emph{parabolic $p$-Laplace equation}, i.e., the evolution equation
\begin{equation}\label{i:eq-parabolic}
\frac{d }{dt}u(t,\mv)+\Lfun_pu(t,\mv)=0,\qquad t\ge 0,\ \mv\in \mV\ ,
\end{equation}
is but a finite dimensional nonlinear dynamical system whose associated Cauchy problem has a unique local solution by the Picard--Lindelöf theorem. In the continuous case, the corresponding parabolic problem
\begin{equation}\label{i:eq-parabolic-omega}
\frac{\partial }{\partial t}u(t,x)+\Delta_pu(t,x)=0,\qquad t\ge 0,\ x\in \Omega\ ,
\end{equation}
on domains $\Omega$ of $\mathbb R^d$ has a long history that goes back to seminal investigations by Ladyzenskaya and others in the 1960s. For $d=1$ it is known that~\eqref{i:eq-parabolic-omega} is nothing but an equivalent formulation of the more popular porous medium equation.

The spaces $\ell^p_{\nu}(\mV)$ (resp., $\ell^p_{\mu}(\mE)$), $p\in[1,\infty]$, are routinely defined for functions supported on $\mV$ with respect to the measure $\nu$ (resp., on $\mE$ with respect to the measure $\mu$). By means of these ``discrete'' Lebesgue spaces, and in view of the analogy between $\grad$ and $\mathcal I^T$ also discrete analogs of Sobolev spaces can eventually be defined, which we denote by  $w^{1,p,2}_{\mu,\nu}(\mV)$.

Well-posedness of the Cauchy problem associated with the parabolic equation~\eqref{i:eq-parabolic} on finite graphs can be refined to also yield a unique global solution for a large class of infinite graphs in a Hilbert space setting: This has been proven by the second author in~\cite{Mug13} and for the reader's convenience we collect some related results in Theorem~\ref{thm:main-new}.

We will state our main results below and postpone their proofs to the main body of the article.

\begin{itemize}
\item \textbf{The generators.} One interesting feature of the theory of linear differential equations on \emph{infinite} graphs is that difference operators, like the $p$-Laplacians in our present setting, may have different self-adjoint extensions: Roughly speaking, in spite of the discrete structure of the environment, ``boundary conditions" (at infinity) may be needed in order to determine a solution uniquely, cf.\ Remark~\ref{rem:car11}.
Two realizations $\Lfun^{\rm D}_p$ and $\Lfun^{\rm N}_p$ of the $p$-Laplacians turn out to be particularly natural.
They are introduced in Definitions~\ref{defi->thm:main} and~\ref{defi-thm:mainnu} and following~\cite{HaeKelLen12} we refer to them as \emph{$p$-Laplacians with Dirichlet} and \emph{Neumann conditions}, respectively. By construction, they coincide whenever applied to functions with finite support, but $\Lfun^{\rm D}_p, \Lfun^{\rm N}_p$ may differ even for $p=2$ on infinite graphs. The associated parabolic $p$-Laplace equations are denoted by \eqref{d:Dirichlet} and \eqref{d:Neumann}, accordingly, and they are governed by strongly continuous semigroups of nonlinear contractions, respectively.

\begin{theo}\label{t:generator}
Let $q\in [1,\infty]$ and $f_0\in \ell^q_\nu(\mV)\cap \ell^2_\nu(\mV)$.  Then $\varphi:t\mapsto e^{-t\Lfun^{\rm N}_p}f_0$ satisfies
\[
\frac{d\varphi}{dt}(t,\mv)=-\mathcal I\left((|\mathcal I^T\varphi(t,\cdot)|^{p-2})\mathcal I^T\varphi(t,\cdot)\right)(\mv),\qquad t\ge 0,\ \mv\in \mV\ ,
\]
and in particular the generator of $(e^{-t\Lfun^{\rm N}_{p}})_{t\ge 0}$ formally acts as a $p$-Laplacian whenever applied to functions with finite support. An analogous assertion holds for the semigroup $(e^{-t\Lfun^{\rm D}_p})_{t\ge 0}$.
\end{theo}

\item \textbf{Time regularity of solutions} Much attention has devoted to the issue of regularity of the solutions to the parabolic $p$-Laplace equation on domains $\Omega$ of $\mathbb R^d$. Galerkin's method yields that the solutions belong to $W^{1,2}(0,\infty; L^2(\Omega))\cap L^{\infty}(0,\infty; W^{1,p}(\Omega))$ and the theory of nonlinear semigroups only ventures into additionally saying that the solution instantaneously belongs to the domain of $\Delta_p$, cf.~\cite[Cap.~3]{Bre73}. Analog assertions hold for the discrete parabolic $p$-Laplace equation.
To the best of our knowledge the issue of higher order time regularity is largely open for general $p\ne 2$, both in the continuous and -- so far -- in the discrete case. This is in striking contrast to the linear case, as the $2$-Laplacian is nothing but the usual (linear) Laplacian, hence the solution to the parabolic problem
\begin{equation}\label{i:eq-parabolic-2}
\frac{d }{dt}u(t,x)+\Delta_2 u(t,x)=0,\qquad t\ge 0,\ x\in \Omega\ ,
\end{equation}
(beware the sign convention in~\eqref{form:plapl}!).
is given by an analytic semigroup on $L^2(\Omega)$: Accordingly, for $p=2$ the solution is analytic in time and, as a consequence of suitable Sobolev embeddings, of class $C^\infty$ in space.

In the case of the parabolic $p$-Laplace equation on graphs, the regularity in the spatial direction is not a natural topic due to the discrete local structure. The issue of time regularity is probably more interesting. We can prove higher order regularity of solutions to~\eqref{i:eq-parabolic} for various $p$.

\begin{theo}\label{resuc2}
 Let $u\in W^{1,2}(0,\infty; \ell^2_\nu(\mV))\cap L^{\infty}(0,\infty; w^{1,p,2}_{\mu,\nu}(\mV))$ be a solution to the parabolic $p$-Laplace equation~\eqref{i:eq-parabolic}, $p\in (1,\infty)$. Then the following assertions are true:  For any $\mv \in \mV,$
\begin{enumerate}[(a)]
\item $u(\cdot,\mv)\in C^{\infty}([0,\infty))$ if $p$ is an even integer,
\item $u(\cdot,\mv)\in C^{p-1,1}([0,\infty))$ if $p$ is an odd integer, and
\item $u(\cdot,\mv)\in C^{\lfloor p\rfloor,p-\lfloor p\rfloor}([0,\infty))$ if $p\in (1,\infty)\setminus\N$.
\end{enumerate}
\end{theo}
Here $\lfloor p\rfloor$ denotes the largest integer less than or equal to $p$.

Furthermore, we can show higher regularity in time of the solutions to~\eqref{i:eq-parabolic} even with respect to the $\ell^q$-norms. The following theorem states that we can get $C^1$ regularity of the time for the semigroups associated with the $p$-Laplace operators for $p\geq 2$. Here and in the following we call a weighted graph $\mG=(\mV,\mE,\nu,\mu)$ \emph{uniformly locally finite} if
\begin{equation}\tag{ULF}\label{a:ULF}
\inf_{\mv\in \mV} \frac{\nu(\mv)}{\mu(\mv)}>0\ ,
\end{equation}
where $\mu(\mv):=\sum \mu(\me)$, summing over all edges $\me$ that have $\mv$ as an endpoint, see Definition~\ref{defi:basic}. The measure $\nu$ is said to be \emph{nondegenerate} if
\begin{equation}\tag{ND}\label{a:nondeg}
\inf_{\mv\in \mV} \nu(\mv)>0\ .
\end{equation}

\begin{theo}\label{thm:higher in norm} Let $\mG$ be uniformly locally finite \eqref{a:ULF} and $\nu$ be nondegenerate \eqref{a:nondeg}. Let $p\geq 2$. Then for any $q\in [1,\infty)$ and any initial data $u_0\in \ell^q_{\nu}(\mV)$ there exists a unique solution $u\in C^1([0,\infty),\ell^q_{\nu}(\mV))$ to the parabolic $p$-Laplace equation.
Furthermore,
\begin{enumerate}[(a)]
\item $u\in C^{\infty}(0,\infty;\ell^q_{\nu}(\mV))$ if $p$ is an even integer,
\item $u\in C^{p-1,1}(0,\infty;\ell^q_{\nu}(\mV))$ if $p$ is an odd integer, and
\item $u\in C^{\lfloor p\rfloor}(0,\infty;\ell^q_{\nu}(\mV))$ if $p\in (2,\infty)\setminus\N.$
\end{enumerate}
\end{theo}
Let us remark the drop in regularity between Theorem~\ref{resuc2} and Theorem~\ref{thm:higher in norm}, if $p$ is a non-integer real number.

Theorem~\ref{thm:higher in norm} is somewhat surprising when compared with known results on regularity of the parabolic $p$-Laplace equations on domains of $\mathbb R^d$: Some Hölder regularity results have been obtained in suitably adapted settings among others in~\cite{DibGiaVes12,Hwa12}, but otherwise little seems to be known. As we will see, the enhanced regularity of solutions stated above actually follows from the discrete nature of graphs.

\item \textbf{Long-time behavior of solutions.} A most remarkable property of the parabolic $p$-Laplace equation on a domain $\Omega\subset\R^n$ is that for $1<p<\frac{2n}{n+1}$ the solutions to its associated Cauchy problem can vanish in finite time (and then, again by usual variational arguments and a parabolic maximum principle, they cannot grow again any more): This phenomenon usually referred in the literature to as ``finite extinction time'' sharply contrasts with the behavior in the linear case of $p=2$. In the case of the continuous $p$-Laplacian on $\mathbb R^d$, this specific property has been proved in~\cite{Bam77,HerVaz81}; by a comparison principle, this result extends to the $p$-Laplacian with Dirichlet boundary conditions on domains.
In the case of the discrete $p$-Laplacian, an analogous result seems to be unknown in this generality. We mention however~\cite[Thm.~4.5]{LeeChu12}, where finite extinction time has been observed for \emph{finite} graphs whenever a Dirichlet condition is imposed on at least one of the nodes of $\mG$ with (non-empty) boundary, or equivalently on finite subgraphs of an infinite graph on which the Dirichlet boundary condition is imposed.


The proof of the finite time extinction property for the parabolic $p$-Laplace equation~\eqref{i:eq-parabolic-omega} on domains presented in~\cite{HerVaz81} boils down to an energy estimate that is essentially based on suitable (dimension depending!) Sobolev inequalities. On graphs, one can derive Sobolev inequalities from the so-called isoperimetric inequalities. Isoperimetric inequalities were originally studied in differential geometry and have proven to be quite useful on weighted graphs, too, see e.g.~\cite{Chu97,Woe00}. They were among the earliest indications that Riemannian geometry can be naturally developed on weighted graphs, too, cf.\ the survey~\cite{Kel14}. For a finite subset $\mV_0$ of $\mV$ we denote by $\partial \mV_0$ the set of all edges with one endpoint in $\mV_0$ and the other one in $\mV\setminus \mV_0$. For some $d>1$ we say that $\mG$ satisfies the \emph{$d$-isoperimetric inequality} if there is a constant $C_d<\infty$ such that
\begin{equation}\tag{$\rm IS_d$}\label{e:isoperimetric constant}
\nu(\mV_0)^{\frac{d-1}{d}}\le  C_d\mu(\partial \mV_0)\qquad \forall\ \mV_0\subset \mV\hbox{ finite},
\end{equation}
where $\nu(\mV_0):=\sum_{\mv\in \mV_0}\nu(\mv)$ and $\mu(\partial\mV_0):=\sum_{\me\in \partial \mV_0}\mu(\me)$, see e.g.\ \cite[Section~4]{Woe00}.

It actually follows from the argument of \cite{HerVaz81} combined with a Sobolev-like inequality on graphs, Theorem~\ref{l:Sobolev} below, that a certain isoperimetric inequality implies the extinction of the solutions to the parabolic $p$-Laplace equation in finite time.

\begin{theo}\label{theo:finiteproof}
Assume $\mG$ to satisfy the \emph{$d$-isoperimetric inequality} \eqref{e:isoperimetric constant} for some $d\ge 2$. If $\mG$ is uniformly locally finite and $\nu$ is nondegenerate, then for all $p\in \left(1,\frac{2d}{d+1}\right)$ and all $f_0\in \ell^{\frac{d(2-p)}{p}}_{\nu}(\mV)$ there is $T_0$ such that the solution $\varphi$ of $\rm(HE^{\rm D}_p)$ with initial data $f_0$ satisfies
\[
\varphi(t,\mv)= 0\qquad \forall\ t\ge T_0\hbox{ and all }\mv\in \mV\ .
\]
The extinction time $T_0$ depends only on $\|f_0\|_{\ell^{\frac{d(2-p)}{p}}_{\nu}(\mV)}$, $C_1$, $d$, $C_d$, and $p$, where
\begin{equation*}C_1:=\inf_{\mv\in \mV} \frac{\nu(\mv)}{\mu(\mv)}>0\ .
\end{equation*}
\end{theo}
It is remarkable that this finite extinction time $T_0$ does not blow up when $p\to 1,$ see Remark~\ref{rem:finite extinction}.

On the contrary, parabolic $p$-Laplace equations $\rm(HE^{\rm N}_p)$ with Neumann conditions enjoy the conservation of mass property for $p\ge 2$. This implies that their solutions never vanish identically for nontrivial initial data.

\begin{theo}\label{thm:conservation of mass}
Let $\mG$ be uniformly locally finite. If $p\geq 2$, then the solution $\varphi$ of $\rm(HE^{\rm N}_p)$ with initial data $f_0\in \ell^{p-1}_{\nu}(\mV)\cap \ell^1_{\nu}(\mV)$ satisfies conservation of mass, i.e., $\|\varphi(t,\cdot)\|_{\ell^1_\nu(\mV)}=\|f_0\|_{\ell^1_\nu(\mV)}$ for all $t\ge 0$.
\end{theo}

Again for $p\geq 2$, we can describe a property that is in some sense dual to finite time extinction, namely eternal activity in each point. It is thus an open question how the solutions behave for $p\in \left(\frac{2d}{d+1},2\right)$.

\begin{theo}\label{prop:irreducibility}
If $p\geq 2$, then the solution $\varphi$ to $\rm(HE^{\rm D}_p)$ satisfies $\varphi(t,\mv)>0$ for all $t>0$ and all $\mv\in \mV$, whenever the initial data $f_0(\mv)\ge 0$ for all $\mv\in \mV$ and $f_0\not\equiv 0$. The same holds for the solution to $\rm(HE^{\rm N}_p)$.
\end{theo}

This property can be interpreted as a kind of strong maximum principle. It is already known that this holds for $p=2$, also for $\rm(HE^{\rm D}_p)$: This is a direct consequence of irreducibility of the semigroup generated by $-\Lfun_2^{\rm N}$ and $-\Lfun_2^{\rm D}$, cf.~\cite[\S~7]{HaeKelLen12} and~\cite[\S~3]{Mug13}. The same behavior has already been remarked in~\cite[Thm.~4.3]{LeeChu12}, in the special case of finite graphs with Dirichlet conditions imposed on a nonempty subset of $\mV$.

Our last result is about parabolic $p$-Laplace equations on graphs of finite measure, i.e.\ finite total $\nu-$measure. We say that a weighted graph $(\mV,\mE,\nu,\mu)$ of finite measure satisfies \emph{the Poincaré inequality} if there exists $C>0$ such that
\begin{equation}\tag{PI}
\|f-\overline{f}\|_{\ell^2_{\nu}(\mV)}\leq C\|\mathcal I^T f\|_{\ell_{\mu}^2(\mE)},
\end{equation}
for any $f$ with finite support,
where $\overline{f}$ denotes the constant function whose value agrees with the mean value of $f$, i.e.
\[
\overline{f}:\equiv \frac{1}{\nu(\mV)}\sum_{\mv\in \mV}f(\mv)\nu(\mv),
\]
where $\nu(\mV):=\sum_{\mv\in \mV}\nu(\mv)$. By the spectral theorem, this is equivalent to the property that
\[
\inf(\sigma(\Lfun_2)\setminus\{0\})>0\ ,
\]
where $\sigma(\Lfun_2)$ is the $\ell^2_{\nu}$-spectrum of the Laplacian $\Lfun_2$. For uniformly locally finite graphs, this can also be reformulated as a Cheeger inequality. If the graph has finite measure and satisfies the Poincaré inequality $\rm(PI)$, then using the conservation of mass property stated in Theorem~\ref{thm:conservation of mass} we obtain the asymptotic behavior of the solution to the Neumann parabolic $p$-Laplace equation for large $p$.

\begin{theo}\label{prop:polynomial}
Let $\mG$ be a uniformly locally finite graph with finite measure and assume $\mG$ to satisfy the Poincaré inequality $\rm(PI)$. Then for any $p\in \left(2,\infty\right)$ and any initial data $f_0\in \ell^2_\nu(\mV)\cap \ell^{p-1}_{\nu}(\mV)$, the solution $\varphi$ of $\rm(HE^{\rm N}_p)$ satisfies
\[
\|\varphi(t,\cdot)-\overline{f_0}\|_{\ell^2_\nu(\mV)}\le Ct^{-\frac{2}{p-2}},\qquad t\ge 0\ ,
\] where $\overline{f_0}=\frac{1}{\nu(\mV)}\sum_{\mv\in\mV}f_0(\mv)\nu(\mv).$\end{theo}
\end{itemize}

The paper is organized as follows: In Section~\ref{sec:2} we recall basic facts on infinite graphs and discrete parabolic $p$-Laplace equations and identify the associated nonlinear semigroups' generators. Section~\ref{sec:3} is devoted to establishing higher regularity in time for the solutions. In Section~\ref{sec:4} we finally prove the key properties of solutions to parabolic $p$-Laplace equations: finite extinction time and conservation of mass.

\section{General setting and the energy functional}\label{sec:2}
As we have already anticipated in the introduction, the object of our investigation is a class of nonlinear dynamical systems. They will be eventually put on stage in the classical framework of simple, non-oriented graphs; but following an idea that goes back to~\cite{Kir45}, it seems that the easiest way to introduce the nonlinear operators that are relevant to us is to assign to each edge an artificial, arbitrary orientation. The reader should however keep in mind that these orientations fulfill the sole goal of providing a convenient parametrization, and in particular of allowing for the introduction of oriented incidence matrices, but play no particular role in the statements of our main results.

We consider throughout a (finite or countable) set $\mV$ and a relation
\[
\mE\subset\{(\mv,\mw)\in \mV\times \mV:\mv\not= \mw\}
\]
such that for any two elements $\mv,\mw\in \mV$ at most one of the pairs $(\mv,\mw),(\mw,\mv)$ belongs to $\mE$: If such an edge exists we denote it by $\mv\mw$ (we stress that $\mv\mw$ and $\mw\mv$ will then be different but equivalent notations for the same edge!)

We refer to the elements of $\mV$ and $\mE$ as \emph{nodes} and \emph{edges}, respectively, that are related by an \emph{oriented} incidence matrix $\mathcal I:=({\iota}_{\mv \me})$ defined by
$${\iota}_{\mv \me}:=\left\{
\begin{array}{ll}
-1 & \hbox{if } \mv \hbox{ is initial endpoint of } \me, \\
+1 & \hbox{if } \mv \hbox{ is terminal endpoint of } \me, \\
0 & \hbox{otherwise}.
\end{array}\right.$$
An edge $\me$ is said to be \emph{incident} in a node $\mv$ if $\iota_{\mv\me}\neq 0$.
Regardless of their orientation, two edges $\me,\mf$ are called \emph{adjacent} if they share an endpoint, i.e., if there exists $\mv\in \mV$ such that $\iota_{\mv\me}\neq 0\neq \iota_{\mv\mf}$.  In the following we always assume connectedness of the graph $\mG$, to avoid trivialities: By this we mean that for any two nodes $\mv,\mw$ a \emph{path} -- i.e., a sequence of adjacent edges -- can always be found that connects $\mv$ to $\mw$. Observe that connectedness is not influenced by the orientation assigned to the edges.

In this paper we call any quadruple $(\mV,\mE,\nu,\mu)$ a \emph{weighted graph}: Here $\mu$ and $\nu$ are weights on $\mE$ and $\mV$, i.e., we are considering two functions $\mu:\mE\to (0,\infty)$ and $\nu: \mV\to (0,\infty).$  That is, we attach a weight $\mu(\me)$, resp.\ $\nu(\mv)$, to each edge $\me$, resp.\ to each node $\mv$. Because $(\mv,\mw)$ and $(\mw,\mv)$ cannot be both edges of the graph, we adopt the notations
\[
\mu_{\mv\mw}=\mu_{\mw\mv}:=\mu(\me)\quad\hbox{and}\quad \mv\sim\mw\qquad \hbox{ if }\me=\mv\mw\in \mE\ .
\]

\begin{defi}\label{defi:basic}
A weighted graph $\mG:=(\mV,\mE,\nu,\mu)$ is called \emph{locally finite (in the measure-theoretical sense)}  if
\[
\mu(\mv):=\sum_{\me \in\mE}|\iota_{\mv\me}|\mu(\me)=\sum_{\substack{\mw\in \mV\\ \mw\sim \mv}}\mu_{\mv\mw}<\infty\qquad \hbox{for all }\mv\in \mV\ .
\]
It is called \emph{uniformly locally finite} if condition~\eqref{a:ULF} introduced in Section~\ref{sec:1} holds.
\end{defi}

For the sake of simplicity, in this paper we only consider connected, locally finite weighted graphs.
\begin{rem}\label{rem:bddincid}
Uniform local finiteness of a weighted graph-- i.e., condition \eqref{a:ULF} -- is nothing but absolute continuity of the node weight $\mu$ with respect to $\nu$: In this case the incidence matrix $\mathcal I$ is bounded from $\ell^p_\mu(\mE)$ to $\ell^p_\nu(\mV)$ (see e.g.~\cite[Lemma~2.9]{Mug14}), for $1\leq p\leq\infty$.
\end{rem}

We will consider the weighted sequence spaces $\ell^p_\mu(\mE)$ and $\ell^q_\nu(\mV)$ defined for $1<p<\infty$ by
\[
\|u\|_{\ell^p_\mu(\mE)}^p := \sum_{\me \in \mE} |u(\me)|^p \mu(\me)\qquad \hbox{and}\qquad \|f\|_{\ell^p_\nu(\mV)}^p:=\sum_{\mv \in \mV} |f(\mv)|^p \nu(\mv),\qquad p\in [1,\infty)\ ,
\]
or
\[
\|u\|_{\ell^\infty_\mu(\mE)}:= \sup_{\me \in \mE} |u(\me)| \mu(\me)\qquad \hbox{and}\qquad \|f\|_{\ell^\infty_\nu(\mV)}:=\sup_{\mv \in \mV} |f(\mv)|\nu(\mv)\ .
\]
These are of course Banach spaces. For simplicity, we usually write $\ell^p_{\nu}$ for $\ell^p_{\nu}(\mV)$ when the space that the functions are defined on is clear in the context.

We are concerned with the theory of function spaces on the measure space $(\mV,\nu).$ It turns out that two geometric (as opposed to: combinatorial) settings are particularly relevant:
\begin{itemize}
\item graphs with \emph{nondegenerate node weight}, i.e., \begin{equation}\label{e:nondegenerate}\inf_{\mv\in \mV}\nu(\mv)>0,\end{equation} in which case we have in particular $\ell^r_\nu(\mV) \hookrightarrow \ell^s_\nu(\mV)$ if $r<s$;
\item graphs of \emph{finite measure}, i.e., $\nu(\mV):=\sum_{\mv\in \mV}\nu(\mv)<\infty$, for which $\ell^r_\nu(\mV) \hookrightarrow \ell^s_\nu(\mV)$ holds if $r>s$.
\end{itemize}
Observe that an infinite graph with nondegenerate node weight necessarily has infinite measure.

\begin{rem}\label{rem:atinfin}
Under the assumption \eqref{e:nondegenerate}, $f\in \ell^p_\nu (\mV)$ for any $1\leq p<\infty$ implies that $f(\mv)\to 0$ as $\dist (\mv,\mv_0)\to \infty$ for some (hence all) $\mv_0\in \mV$. This noteworthy property will prove useful later.\end{rem}

In order to perform the analysis of the parabolic $p$-Laplace equation we begin by defining discrete Sobolev spaces: For $1\leq p\leq \infty$ we let
\begin{equation}\label{eq:discsob}
w^{1,p,2}_{\mu,\nu}(\mV):=\left\{f\in \ell^2_\nu(\mV):\mathcal I^T f\in \ell^p_\mu(\mE)\right\}\ ,
\end{equation}
equipped with the norm defined by
\[
\|f\|_{w^{1,p,2}_{\mu,\nu}}:=\|f\|_{\ell^2_\nu(\mV)}+\|\mathcal I^T f\|_{\ell^p_\mu(\mE)}\ .
\]
In other words,
$$w^{1,p,2}_{\mu,\nu}(\mV)=\left\{f\in \ell^2_\nu(\mV):\|\mathcal I^T f\|^p_{\ell^p_\mu}=\sum_{\me\in \mE} \mu(\me) |f(\me_{+})-f(\me_{-})|^p<+\infty\right\},$$
where for any (oriented!) edge $\me$ we denote by $\me_{+}$ and $\me_{-}$ the initial and terminal endpoint of $\me$, respectively, i.e., $\me=(\me_{-},\me_{+})$. Furthermore, we denote by
\[
\mathring{w}^{1,p,2}_{\mu,\nu}(\mV)
\]
the closure in $w^{1,p,2}_{\mu,\nu}(\mV)$ of the space $c_{00}(\mV)$ of finitely supported functions on $\mV$.

\begin{defi}\label{defi->thm:main}
The operator $\Lfun^{\rm N}_p$ is defined by
\[
\begin{split}
D(\Lfun^{\rm N}_p)&:=\bigg\{f\in w^{1,p,2}_{\mu,\nu}(\mV):\exists g\in \ell^2_\nu(\mV)\\
&\qquad \hbox{ s.t.\ }\sum_{\me\in \mE} \mu(\me)|({\mathcal I^T}f)(\me)|^{p-2}({\mathcal I^T}f)(\me) ({\mathcal I^T}h)(\me)=\sum_{\mv \in \mV} g(\mv)h(\mv)\nu(\mv)\;\; \forall h \in w^{1,p,2}_{\mu,\nu}(\mV)\bigg\}\ , \\
\Lfun^{\rm N}_pf &:=g\ .
\end{split}
\]
Let $p\in (1,\infty)$. If $f_0\in w^{1,p,2}_{\mu,\nu}(\mV)$, then a \emph{solution} to the Cauchy problem
\begin{equation}\label{d:Neumann}
\tag{HE$^{\rm N}_p$}
\left\{
\begin{array}{rcll}
\dot{\varphi}(t,\mv)&=&-\Lfun^{\rm N}_p \varphi(t,\mv),\qquad &t>0,\; \mv\in \mV, \\
\varphi(0,\mv)&=&f_0(\mv),&\mv\in \mV,
\end{array}
\right.\end{equation}
is a function $\varphi\in W^{1,2}(0,\infty;\ell^2_\nu(\mV))\cap L^\infty(0,\infty;w^{1,p,2}_{\mu,\nu}(\mV))$ for which
\begin{itemize}
\item $\varphi(t,\cdot)\in D(\Lfun^{\rm N}_p)$ for a.e.\ $t>0$,
\item $\dot{\varphi}(t,\mv)=-\Lfun^{\rm N}_p \varphi(t,\mv)$ is satisfied for a.e.\ $t>0$, and
\item $\varphi(0,\cdot)=f_0$.
\end{itemize}
\end{defi}

\begin{defi}\label{defi-thm:mainnu}
The operator $\Lfun^{\rm D}_p$ is defined by
\[
\begin{split}
D(\Lfun^{\rm D}_p)&=\bigg\{f\in \mathring{w}^{1,p,2}_{\mu,\nu}(\mV):\exists g\in \ell^2_\nu(\mV)\\
&\qquad \hbox{ s.t.\ }\sum_{\me\in \mE} \mu(\me)|({\mathcal I^T}f)(\me)|^{p-2}({\mathcal I^T}f)(\me) ({\mathcal I^T}h)(\me)=\sum_{\mv \in \mV} g(\mv)h(\mv)\nu(\mv)\;\; \forall h \in \mathring{w}^{1,p,2}_{\mu,\nu}(\mV)\bigg\}\ , \\
\Lfun^{\rm D}_p f&:=g\ 	.
\end{split}
\]
Let $p\in (1,\infty)$. If $f_0\in \mathring{w}^{1,p,2}_{\mu,\nu}(\mV)$, then a \emph{solution} to the Cauchy problem
\begin{equation}\label{d:Dirichlet}
\tag{HE$^{\rm D}_p$}
\left\{
\begin{array}{rcll}
\dot{\varphi}(t,\mv)&=&-\Lfun^{\rm D}_p \varphi(t,\mv),\qquad &t>0,\; \mv\in \mV, \\
\varphi(0,\mv)&=&f_0(\mv),&\mv\in \mV,
\end{array}
\right.\end{equation}
is a function $\varphi\in W^{1,2}(0,\infty;\ell^2_\nu(\mV))\cap L^\infty(0,\infty;\mathring{w}^{1,p,2}_{\mu,\nu}(\mV))$ for which
\begin{itemize}
\item $\varphi(t,\cdot)\in D(\Lfun^{\rm D}_p)$ for a.e.\ $t>0$,
\item $\dot{\varphi}(t,\mv)=-\Lfun^{\rm D}_p \varphi(t,\mv)$ is satisfied for a.e.\ $t>0$, and
\item $\varphi(0,\cdot)=f_0$.
\end{itemize}
\end{defi}

\begin{rem}\label{rem:car11}
Finding conditions on $\mG$ under which the identity
\[
\Lfun^{\rm D}_p=\Lfun^{\rm N}_p
\]
holds is interesting, since in this case the parabolic $p$-Laplace equation enjoys uniqueness of solution without need to impose further conditions at infinity.
In view of Remark~\ref{rem:bddincid} $\mathcal I$ is bounded from $\ell^p_\mu(\mE)$ to $\ell^p_{\nu}(\mV)$ for all $p\in [1,\infty]$, provided $\mG$ is uniformly locally finite. Accordingly
\[
\mathring{w}^{1,p,2}_{\mu,\nu}(\mV)=w^{1,p,2}_{\mu,\nu}(\mV)=\ell^2_\nu(\mV)
\]
and therefore $\Lfun^{\rm D}_p=\Lfun^{\rm N}_p$ if $\mG$ satisfies~\eqref{a:ULF} and additionally any of the following conditions holds:
\begin{itemize}
\item $p=2$,
\item $\nu$ is nondegenerate and $p\in (2,\infty)$
\item $\nu\in \ell^\infty(\mV)$ and $p\in [1,2)$.
\end{itemize}
Alternatively, it has been proved in~\cite{Jor08,Tor10,KelLen12} that the Laplacian is essentially self-adjoint, hence $\Lfun^{\rm D}_2=\Lfun^{\rm N}_2$ if
\begin{itemize}
\item $p=2$ and $\nu$ is nondegenerate.
\end{itemize}
\end{rem}

The following summarizes some well-posedness results proved in~\cite[\S~3]{Mug13}. We adopt the terminology introduced in Definitions~\ref{defi->thm:main}--~\ref{defi-thm:mainnu}.

\begin{theo}\label{thm:main-new}
Let $p\in (1,\infty)$. Then the following assertions hold.
\begin{enumerate}[(1)]
\item The solution to $\rm(HE^{\rm D}_p)$ (resp., to $\rm(HE^{\rm N}_p)$) is given by a strongly continuous semigroup  of nonlinear contractions on $\ell^2_\nu(\mV)$ that we denote by  $(e^{-t\Lfun^{\rm N}_{p}})_{t\ge 0}$ (resp., by $(e^{-t\Lfun^{\rm D}_{p}})_{t\ge 0}$). If in particular $f_0\in \ell^2_\nu(\mV)$, then there is a unique $\varphi\in C([0,\infty);\ell^2_\nu(\mV))$ that satisfies the initial condition, that admits for all $t>0$ a right derivative and such that its right derivative satisfies the differential equation for all $t>0$.
\item Let $f_0\in {w}^{1,p,2}_{\mu,\nu}(\mV)$ (resp., $f_0\in \mathring{w}^{1,p,2}_{\mu,\nu}(\mV)$). Then there exists a unique $\varphi\in W^{1,2}(0,\infty;\ell^2_\nu(\mV))\cap L^\infty(0,\infty;w^{1,p,2}_{\mu,\nu}(\mV))$ (resp., $\varphi\in W^{1,2}(0,\infty;\ell^2_\nu(\mV))\cap L^\infty(0,\infty;\mathring{w}^{1,p,2}_{\mu,\nu}(\mV))$) that lies pointwise in $D(\Lfun^{\rm N}_p)$ (resp., in $D(\Lfun^{\rm D}_p)$) and satisfies the differential equation for a.e.\ $t>0$ as well as the initial condition.
\end{enumerate}
\end{theo}

For a node subset $\mW\subset \mV$, the \emph{subgraph $\mG[\mW]$ of $\mG$ induced by $\mW$} is by definition the graph whose node set is $\mW$ and whose edge set consists of all edges in $\mE$ whose endpoints are both in $\mW$. If $\mG$ is a \emph{weighted} graph with edge weights $\mu$ and node weights $\nu$, then each edge and each node in the induced subgraph keeps the same weight it has in $\mG$.
\begin{defi}
Let $\mG$ be a graph. A family of graphs $(\mG_n)_{n\in \mathbb N}$ is called \emph{growing} if $\mG_n$ is an induced subgraph of $\mG_m$ for all $n,m\in \mathbb N$ with $n\le m$. It is said to \emph{exhaust $\mG$} if
\begin{itemize}
\item $\mG_n$ is an induced subgraph of $\mG$  for all $n\in \mathbb N$  and
\item $\bigcup_{n\in \mathbb N}\mV_n =\mV$.
\end{itemize}
\end{defi}

Then~\cite[Thm.~3.4]{Mug13} also implies the following.
\begin{theo}\label{thm:main-2}
Let $p\in (1,\infty)$. Then the following assertions hold.
\begin{enumerate}[(1)]
\item For all $f_0\in \mathring{w}^{1,p,2}_{\mu,\nu}(\mV)$ there is a growing family of finite graphs $(\mG_n)_{n\in \mathbb N}$ that exhausts $\mG$ and such that the sequence of solutions $(\varphi_n)_{n\in \mathbb N}$ to the Cauchy problem
\begin{equation}
\tag{HE$^{(n)}_p$}
\left\{
\begin{array}{rcll}
\dot{\varphi}_n(t,\mv)&=&-\Lfun^{(n)}_p \varphi_n(t,\mv),\qquad &t\ge 0,\; \mv\in \mV_n, \\
\varphi_n(t,\mv)&=&0, &t\ge 0,\; \mv \in \mV\setminus\mV_n, \\
\varphi_n(0,\mv)&=&f_0(\mv), &\mv\in \mV_n,
\end{array}
\right.\end{equation}
where $\mV_n$ denotes the node set of $\mG_n$ and for all $h\in c_{00}(\mV)$	
\[
\begin{split}
\Lfun^{(n)}_p h(\mv)&:=\left\{
\begin{array}{ll}
\frac{1}{\nu(\mv)}\sum\limits_{\substack{\mw\in \mV_n\\ \mw \sim \mv}} \mu(\mv,\mw)\left|h(\mv)-h(\mw)\right|^{p-2}\left(h(\mv)-h(\mw)\right)\\
 \qquad + \frac{1}{\nu(\mv)}\left|h(\mv)\right|^{p-2}h(\mv)\sum\limits_{\substack{\mw\not\in \mV_n\\ \mw \sim \mv}} \mu(\mv,\mw),\qquad & \mv \in \mV_n, \\
 \\
-\frac{1}{\nu(\mv)}\left|h(\mv)\right|^{p-2}\left(h(\mv)\right)\sum\limits_{\substack{\mw\in \mV_n\\ \mw \sim \mv}} \mu(\mv,\mw),&\mv\not\in \mV_n,
\end{array}
\right.
\end{split}
\] converges to the unique solution $\varphi$ of $\rm(HE^{\rm D}_p)$, weakly in $W^{1,2}(0,\infty;\ell^2_\nu(\mV))$ and weakly$^*$ in $L^\infty(0,\infty;\mathring{w}^{1,p,2}_{\mu,\nu}(\mV))$.
\item If $f_0\in \ell^q_\nu(\mV)$ for some $q\in [1,2)$, then $(\varphi_n)_{n\in \mathbb N}$ converges to $\varphi$ also weakly$^*$ in $L^\infty(0,\infty;\ell^q_\nu(\mV))$.
\end{enumerate}
\end{theo}

Let us recall the following result, which is~\cite[Cor.~3.11]{Mug13}.
\begin{cor}\label{ralphwell-local}
Let $p\in (1,\infty)$. Both semigroups $(e^{-t\Lfun^{\rm N}_p})_{t\ge 0}$ and $(e^{-t\Lfun^{\rm D}_p})_{t\ge 0}$ extrapolate to a family of nonlinear semigroups on $\ell^q_{\nu}(\mV)$ for all $q\in [1,\infty]$.
\end{cor}
This result is proved by interpolation methods. Accordingly, it is known that for each $q\in [1,\infty]$ there exists an operator acting on $\ell^q_\nu(\mV)$ that agrees with $e^{-t\Lfun^{\rm N}_p}f$ whenever $f\in \ell^2_\nu(\mV)\cap \ell^q_\nu(\mV)$, and that these operators form a family $(e^{-t\Lfun^{\rm N}_{p,q}})_{t\ge 0}$, which is in fact a strongly-continuous semigroup of nonlinear contractions on $\ell^q_\nu(\mV)$. (Same holds for $(e^{-t\Lfun^{\rm D}_p})_{t\ge 0}$.) However, it is a priori not clear which operators generate such semigroups -- this is a problem that generally appears even whenever \emph{linear} semigroups are defined by extrapolation/interpolation. However, we are able to describe their generators as in Theorem~\ref{t:generator}.

\begin{proof}[Proof of Theorem~\ref{t:generator}]
Consider a counting of $\mV$, say $\mV=\{\mv_1,\mv_2,\ldots\}$, and denote by $\mathbf 1_m$ the characteristic function of the finite set $\{\mv_1,\ldots,\mv_m\}$. Take $f\in \ell^q_\nu(\mV)$ and let $f_n:=f\cdot {\mathbf 1}_n$, so that $(f_n)_{n\in \mathbb N}\subset c_{00}(\mV)$. If $f\ge 0$, then $0\le f_n\le f_{n+1}\le f$ for all $n\in \mathbb N$ and $f=\ell^q_\nu-\lim_{n\to \infty}f_n$. Since $f_n\in \ell^2_\nu(\mV)\cap \ell^q_\nu(\mV)$ for all $n\in \mathbb N$,
\[
e^{-t\Lfun^{\rm N}_{p,q}}f_n=e^{-t\Lfun^{\rm N}_{p}}f_n\ ,
\]
and by the definition of $e^{-t\Lfun^{\rm N}_{p}}$ the mapping $t\mapsto e^{-t\Lfun^{\rm N}_p}f_n$ satisfies the wished differential equation. Because $(e^{-t\Lfun^{\rm N}_p})_{t\ge 0}$ is order preserving,
\[
0\le e^{-t\Lfun^{\rm N}_{p,q}}f_n=e^{-t\Lfun^{\rm N}_{p}}f_n\le e^{-t\Lfun^{\rm N}_{p}}f_{n+1}=e^{-t\Lfun^{\rm N}_{p,q}}f_{n+1}.
\]
By Beppo Levi's monotone convergence theorem their limit $e^{-t\Lfun^{\rm N}_{p,q}}f$ (which is unique and hence has to agree with $e^{-t\Lfun^{\rm N}_{p,q}}f$) belongs to $\ell^q_\nu(\mV)$. Because $t\mapsto e^{-t\Lfun^{\rm N}_{p}} f_n$ satisfies the differential equation for each node in $\{\mv_1,\cdots,\mv_n\}$, so does the limit $t\mapsto e^{-t\Lfun^{\rm N}_{p,q}} f$ for each node in $\mV$.
The case of general $f$ follows again approximating $f$ by its truncations with finite support, using this time Lebesgue's dominated convergence theorem to show that the limit belongs to $\ell^q_\nu(\mV)$.

The Dirichlet case follows from the same argument.
\end{proof}

\section{Higher order time regularity and nonlinear semigroups}\label{sec:3}
In the previous section we have recalled some results on the  existence of solutions. In this section we will show that these solutions are in fact strong.
\subsection{Pointwise time regularity.}
Theorem~\ref{thm:main-new} states existence and uniqueness solutions to the parabolic $p$-Laplace equations, both under Dirichlet and Neumann condition and for all $p\in (1,\infty)$, in a suitably week sense. In the theory of partial differential equations it is often routine to obtain weak solutions; the main difficulty consists in showing that these weak solutions in fact satisfy the relevant PDE \emph{pointwise} rather than merely in the integrated sense: this is usually done by some tricky regularity arguments.

In the following we take advantage of the discreteness of the underlying spaces and are thus able to show some new regularity results for weak solutions without employing any deep tools in analysis. In this way, we can prove that the constructed weak solution $u(\cdot,\mv)$ belongs to the Hölder space $C^{\lfloor p\rfloor,p-\lfloor p\rfloor}([0,\infty))$ for any fixed $\mv\in \mV$ and it pointwise satisfies the parabolic $p$-Laplace equation \eqref{i:eq-parabolic} for any $t>0$ (as opposed to: for a.e.\ $t>0$, as yielded by Theorem~\ref{thm:main-new}) where
\begin{equation}\label{e:p laplacian}\Lfun_p u(\mv)=-\frac{1}{\nu(\mv)}\sum_{\substack{\mw\in\mV \\ \mw\sim\mv}}\mu_{\mw\mv}|u(\mw)-u(\mv)|^{p-2}(u(\mw)-u(\mv))\ .\end{equation}

\begin{proof}[Proof of Theorem~\ref{resuc2}] It follows from Theorem~\ref{thm:main-new} that for any fixed $\mv \in \mV$ the equation
\[
D_t^{+}u(\cdot,\mv)=-\Lfun_p u(\cdot,\mv)\in C([0,\infty))	
\]
 holds pointwise.

Now, recall that if a function $f:[0,\infty)\to X$ is continuous and right differentiable at each point, with the right derivative function being continuous, then $f$ is differentiable, see e.g.~\cite[Cor.~2.1.2]{Paz83} if $X$ is a finite dimensional vector space, or~\cite[Lemma~8.9]{HunMacMey13} if $X$ is a general Banach space.
Accordingly, $u(\cdot,\mv)\in C^1([0,\infty))$ for any $\mv\in \mV$ and in fact one has
\begin{equation}\label{e:first derivative}
D_t u(t,\mv)=-\Lfun_p u(t,\mv)=\frac{1}{\nu(\mv)}\sum_{\mw\sim \mv}\mu_{\mv\mw} g_p\left(u(t,\mw)-u(t,\mv)\right)=:U(t,\mv)\ ,
\end{equation}
where
\begin{equation}\label{e:gp}
g_p(\alpha):=|\alpha|^{p-2}\alpha,\quad \alpha\in\R.
\end{equation}
It can be checked directly that
\[g_p(\cdot)\in\left\{\begin{array}{ll}
C^{\infty}(\R),& \rm{if}\ p\ \rm{is\ an\ even\ integer}, \\
C^{p-2,1}(\R),& \rm{if}\ p\ \rm{is\ an\ odd\ integer}, \\
C^{\lfloor p-1\rfloor, p-\lfloor p\rfloor}(\R),& \rm{otherwise}. \\
\end{array} \right.\]

With the previous observation on the regularity of $g_p$ and the fact $u(\cdot,\mv)\in C^1([0,\infty))$, the theorem then follows from standard bootstrap arguments.
\end{proof}

\subsection{Time regularity in norm.}
We want to show that for any $2\leq p<\infty$ there exists a semigroup associated with the $p$-Laplacian in $\ell^q_{\nu}, q\in [1,\infty)$, which is smooth in time with respect to the $\ell^q_{\nu}$ norm in space, i.e., for any $u_0\in \ell^q_{\nu}$ we have $u(t,\cdot)\in \ell^q_{\nu}$ and $u\in C^1([0,\infty); \ell^q_{\nu})$.

We need some lemmata. The following is a direct consequence of~\cite[Cor.~3.11]{Mug13}.
\begin{lemma}\label{l:monotone} Let $\varphi$ be the solution to the parabolic $p$-Laplace equation with Dirichlet or Neumann conditions and initial data $f_0\in \ell^q_{\nu}$, $q\in [1,\infty]$. Then $\varphi(t,\cdot)\in \ell^q_{\nu}$ and
\begin{equation}\label{e:monotone}\|\varphi(t,\cdot)\|_{\ell^q_{\nu}}\leq \|f_0\|_{\ell^q_{\nu}}.\end{equation}
\end{lemma}

\begin{lemma}\label{l:normalized operator}
Let $\mG$ be uniformly locally finite. Then for any $p\in(1,\infty),$ $q\in[1,\infty]$ and any $u\in \ell^q_{\nu},$ we have $\Lfun_p u\in \ell^{\frac{q}{p-1}}_{\nu}$ and \begin{equation}\label{e:normalized} \|\Lfun_p u\|_{\ell^{\frac{q}{p-1}}_{\nu}}\leq C \|u\|_{\ell^q_{\nu}}^{p-1}\end{equation} where $C=C(K,p,q)$ and $K:=\inf_{\mv\in V}\frac{\nu(v)}{\mu(v)}>0$.
Moreover, if $p\geq 2$ and $\inf_{\mv\in V}\nu(v)>0,$ then $\Lfun_p$ maps $\ell^q_{\nu}(\mV)$ to $\ell^q_{\nu}(\mV)$ for any $q\in [1,\infty]$.
\end{lemma}
\begin{proof} Set $s:=\frac{q}{p-1}.$ We only prove for the case $q<\infty$. (The case $q=\infty$ is easy).
By Jensen's inequality,
\begin{eqnarray*}|\Lfun_p u(\mv)|^s&\leq& \left(\frac{\mu(\mv)}{\nu(\mv)}\right)^s\left(\sum_{\substack{\mw\in \mV\\ \mw \sim \mv}}\frac{\mu_{\mw\mv}}{\mu(\mv)}|u(\mw)-u(\mv)|^{p-1}\right)^s\\
&\leq& \left(\frac{\mu(\mv)}{\nu(\mv)}\right)^s\sum_{\substack{\mw\in \mV\\ \mw \sim \mv}}\frac{\mu_{\mw\mv}}{\mu(\mv)}|u(\mw)-u(\mv)|^{q}\ .\end{eqnarray*}
Hence
\begin{eqnarray*}\|\Lfun_p u\|_{\ell^s_{\nu}}^s&\leq&C\sum_{\mv\in\mV}\left(\frac{\mu(\mv)}{\nu(\mv)}\right)^{s-1}\sum_{\substack{\mw\in \mV\\ \mw \sim \mv}}\mu_{\mw\mv}(|u(\mw)|^q+|u(\mv)|^q)\\
&=&C(K,q)\sum_{\mv\in\mV}\mu(\mv)|u(\mv)|^q\leq C\sum_{\mv\in\mV} \nu(\mv)|u(\mv)|^q\ .
\end{eqnarray*}
This proves the first assertion. The second assertion is a direct consequence.\end{proof}

Let $u$ be a pointwise solution to the parabolic $p$-Laplace equation.
Then
$$D_t u(t,\mv)=\frac{1}{\nu(\mv)}\sum_{\mw\in \mV}\mu_{\mv\mw}g(u(t,\mw)-u(t,\mv))\ ,$$  where $g(\cdot):=g_p(\cdot)$ is defined in \eqref{e:gp}.
For convenience, given $\mv,\mw\in \mV,$ we set $h(t):=h_{\mv\mw}(t)=u(t,\mw)-u(t,\mv).$
Then \begin{equation}\label{e:higher order of u}\frac{d^n}{dt^n}u(t,\mv)=\frac{1}{\nu(\mv)}\sum_{\mw\in \mV}\mu_{\mv\mw}\frac{d^n}{dt^n}g(h_{\mv\mw}(t)).\end{equation}
For brevity, we denote $f^{(n)}(t)=\frac{d^n}{dt^n}f(t)$ the $n$-th derivative of a function $f.$
By the formula of high order derivatives of composite functions, one has
\begin{equation}\label{e:higher order derivatives} \frac{d^n}{dt^n}f(g(t))=\sum_{\substack{k_1,k_2,\cdots,k_n\in \Z_{\geq 0}, \\ \sum_{i=1}^nik_i=n}} \frac{n!}{k_1!k_2!\cdots k_n!}g^{(k)}(h(t)) \left(\frac{h'(t)}{1!}\right)^{k_1}\left(\frac{h''(t)}{2!}\right)^{k_2}\cdots \left(\frac{h^{(n)}(t)}{n!}\right)^{k_n},
\end{equation} where $k=\sum_{i=1}^nk_i.$

Now we are ready to prove the higher order regularity of time in the norm.

\begin{proof}[Proof of Theorem~\ref{thm:higher in norm}]
For $u_0\in \ell^q_\nu(\mV)$, by Theorem~\ref{thm:main-new} and Corollary~\ref{ralphwell-local} there exists a unique solution $u\in C([0,\infty);\ell^q_{\nu}).$
By Theorem~\ref{resuc2}, $u(\cdot, \mv)\in C^1([0,\infty))$ for any $\mv\in \mV$ and by Lemma~\ref{l:monotone} $\|u(t,\cdot)\|_{\ell^q_{\nu}}\leq \|u_0\|_{\ell^q_{\nu}}.$ Moreover, by the assumptions of $\mu$ and $\nu,$ Lemma~\ref{l:normalized operator} implies that the right hand side of the parabolic $p$-Laplace equation $\Lfun_p u\in \ell^{q}_{\nu}$ and by the interpolation inequality, \begin{eqnarray*}\left\|\frac{d}{dt}u(t,\cdot)\right\|_{\ell^q}&=&\|\Lfun_p u(t,\cdot)\|_{\ell^q_{\nu}}\leq C\|u(t,\cdot)\|_{\ell^{q(p-1)}_{\nu}}^{p-1}\\&\leq&  C\|u_0\|_{\ell^{q(p-1)}_{\nu}}^{p-1}
\leq C\|u_0\|_{\ell^{q}_{\nu}}^{p-1}
\end{eqnarray*}
which is uniform in $t\in \mathbb R_+$. This implies that $u\in C^1([0,\infty);\ell^q_{\nu}).$

For the higher order estimate, we follow the same argument as in Theorem~\ref{resuc2} with uniform norm estimates. For convenience, we set $M:=\|u_0\|_{\ell^\infty_{\nu}}$ and $W_q:=\|u_0\|_{\ell^{q}_{\nu}}$ for $q\in [1,\infty).$
By the equations \eqref{e:higher order of u} and \eqref{e:higher order derivatives}, an easy induction argument implies that for any $p\in (1,\infty)$ \begin{equation}\label{e:higher u_t infty}\left\|\frac{d^{k}}{dt^k}u(t,\cdot)\right\|_{\ell^{\infty}_{\nu}}\leq C(\|u(0,\cdot)\|_{\ell^{\infty}_{\nu}}), \ \ \ \forall\ k\leq \lfloor p\rfloor-1.\end{equation} This is even true for any $k\in\N$ if $p$ is an even integer.

We now proceed to prove the regularity in $\ell^q_\nu$-norm for any $q\in[1,\infty)$. We divide the proof into three cases:

{\bf Case 1.} $p$ is even. We want to show that $u\in C^{n}([0,\infty);\ell^q_{\nu})$ for any $n\in \N$ and $$\left\|\frac{d^{m+1}}{dt^{m+1}}u(t,\cdot)\right\|_{\ell^q_{\nu}}\leq C(p,q,m,W_q),\ \  \forall\ 0\leq m\leq n-1.$$ This is true for $n=1$ by the previous result. By the induction on $n,$ suppose this holds for $n,$ we want to show it is true for $n+1.$
By taking the $(n+1)-$th derivative of $u$ in $t$, one has
\begin{equation}\label{e:high order composite}
\frac{d^{n+1}}{dt^{n+1}}u(t,\mv)=\frac{1}{\nu(\mv)}\sum_{\mw\in \mV}\mu_{\mv\mw}\sum_{\substack{k_1,k_2,\cdots,k_n\in \Z_{\geq 0}, \\ \sum_{i=1}^nik_i=n}} C(k_1,\cdots,k_n)g^{(k)}(h_{\mv\mw}(t)) \left(h_{\mv\mw}'(t)\right)^{k_1}\cdots (h_{\mv\mw}^{(n)}(t))^{k_n},
\end{equation} where $k=\sum_{i}k_i.$
Note that for even integer $p,$ \[|g_p^{(l)}(\alpha)|=\left\{\begin{array}{ll}
C|\alpha|^{p-1-l},&0\leq l\leq  p-1, \\
0,& p\leq l, \\
\end{array} \right.\]and by \eqref{e:higher u_t infty}, $|h_{\mv\mw}^{(k)}(t)|\leq C(M)$ for any $k\leq n.$
This yields
\begin{eqnarray*}
\left|\frac{d^{n+1}}{dt^{n+1}}u(t,\mv)\right|\leq \frac{C}{\nu(\mv)}\sum_{\mw\in \mV}\mu_{\mv\mw}(|h_{\mv\mw}'(t)|+\cdots +|h_{\mv\mw}^{(n)}(t)|).
\end{eqnarray*}
Hence because $|h_{\mv\mw}^{(k)}(t)|\leq |\frac{d^k}{dt^k}u(t,\mw)|+|\frac{d^k}{dt^k}u(t,\mv)|$ one has
\begin{eqnarray*}\left\|\frac{d^{n+1}}{dt^{n+1}}u(t,\cdot)\right\|_{\ell^q_{\nu}}^q&\leq& C\sum_{k=1}^n\sum_{\mv\in \mV}\left(\frac{\mu(\mv)}{\nu(\mv)}\right)^{q-1}\sum_{\mw\in \mV}\mu_{\mv\mw}\left(\left|\frac{d^k}{dt^{k}}u(t,\mw)\right|^{(p-1)q}+\left|\frac{d^k}{dt^{k}}u(t,\mv)\right|^{(p-1)q}\right)\\
&\leq& C\sum_{k=1}^n\left\|\frac{d^k}{dt^{k}}u(t,\cdot)\right\|_{\ell^q_\nu}^q\leq C(p,q,n,M,W_q),\end{eqnarray*} where we have used
\eqref{e:higher u_t infty} for the case of even $p$ and the induction assumptions. This proves the result.

{\bf Case 2.} $p=2l+1$ for some $l\in \N.$ Set $n=2l-1.$ By the same argument as before, one can show that
\begin{equation}\label{e:estimates for q}\left\|\frac{d^{m}}{dt^{m}}u(t,\cdot)\right\|_{\ell^q_{\nu}}^q\leq C(p,q,M,W_q), \ \ \ \forall\ m\leq n+1.\end{equation} Now we want to show the Lipschitz regularity of the $(2l)$-th derivative in time of the solution with respect to the norm of $\ell^q_\nu$.
For brevity, we denote by $f(\cdot)|_{t_1}^{t_2}:=f(t_2)-f(t_1)$ the difference of the function $f$ at $t_1$ and $t_2$.
A direct calculation shows that
\begin{eqnarray*}&&
  \left|\frac{d^{n+1}}{dt^{n+1}}u(\cdot,\mv)|_{t_1}^{t_2}\right|\\&\leq&\frac{C}{\nu(\mv)}\sum_{\mw\in \mV}\mu_{\mv\mw}\sum_{\substack{k_1,k_2,\cdots,k_n\in \Z_{\geq 0}, \\ \sum_{i=1}^nik_i=n}} \left|g^{(k)}\circ h_{\mv\mw}(\cdot) \left(h_{\mv\mw}'(\cdot)\right)^{k_1}\cdots (h_{\mv\mw}^{(n)}(\cdot))^{k_n}|_{t_1}^{t_2}\right|\\
  &=&\frac{C}{\nu(\mv)}\sum_{\mw\in \mV}\mu_{\mv\mw}\sum_{\substack{k_1,k_2,\cdots,k_n\in \Z_{\geq 0}, \\ \sum_{i=1}^nik_i=n}}\left|  g^{(k)}\circ h_{\mv\mw}(\cdot)|_{t_1}^{t_2} \left(h_{\mv\mw}'(t_2)\right)^{k_1}\cdots (h_{\mv\mw}^{(n)}(t_2))^{k_n}\right. \\&&\left.+\cdots+g^{(k)}\circ h_{\mv\mw}(t_1) \left(h_{\mv\mw}'(t_1)\right)^{k_1}\cdots (h_{\mv\mw}^{(n)}(\cdot))^{k_n}|_{t_1}^{t_2}\right|\\
  &\leq& \frac{C(M)}{\nu(\mv)}\sum_{\mw\in \mV}\mu_{\mv\mw}\sum_{m=0}^n \left|h^{(m)}_{\mv\mw}(\cdot)|_{t_1}^{t_2}\right|\\
  &\leq& \frac{C(M)}{\nu(\mv)}\sum_{\mw\in \mV}\mu_{\mv\mw}\sum_{m=0}^n \left(|u^{(m)}(\cdot,\mv)|_{t_1}^{t_2}|+|u^{(m)}(\cdot,\mw)|_{t_1}^{t_2}|\right)\\
  &\leq&  \frac{C(M)}{\nu(\mv)}\sum_{\mw\in \mV}\mu_{\mv\mw}\sum_{m=0}^n \left(\int_{t_1}^{t_2}|u^{(m+1)}(\tau,\mv)|d\tau+\int_{t_1}^{t_2}|u^{(m+1)}(\tau,\mw)|d\tau\right)
\end{eqnarray*} where we have used the estimate \eqref{e:higher u_t infty}.
Now we have
	\begin{eqnarray*}
	  \left\|\frac{d^{n+1}}{dt^{n+1}}u(\cdot,\mv)|_{t_1}^{t_2}\right\|_{\ell^q_{\nu}}^q&\leq &C|t_2-t_1|^{q-1}\sum_{\mv,\mw\in \mV}\mu_{\mv\mw}\sum_{m=0}^n(\int_{t_1}^{t_2}|u^{(m+1)}(\tau,\mv)|^qd\tau+\int_{t_1}^{t_2}|u^{(m+1)}(\tau,\mw)|^qd\tau)\\
  &\leq&C|t_2-t_1|^{q-1}\sum_{m=0}^n\int_{t_1}^{t_2}\|u^{(m+1)}(\tau,\mv)\|_{\ell^q_{\nu}}^q d\tau\\
  &\leq& C(p,q,M,W_q)|t_2-t_1|^q\ ,
\end{eqnarray*} where the last inequality follows from the estimates~\eqref{e:estimates for q}.

{\bf Case 3.} $p\in (2,\infty)\setminus \N.$ This follows from the same argument as in Case 1 by noting that $|g_p^{(l)}(\alpha)|\leq C|\alpha|^{p-1-l}$ for any integer $0\leq l\leq p-1.$
\end{proof}

	\label{sec:4}

In this section we will investigate more closely the asymptotics of solutions to the $p$-Laplacian heat equation as $t\to \infty$.

\subsection{Finite extinction time}

Before proving Theorem~\ref{theo:finiteproof} we need two auxiliary results.

To begin with, by using isoperimetric inequalities we can prove a discrete analog of the usual Sobolev inequalities, thus slightly generalizing~\cite[Prop.~4.3]{Woe00}.
\begin{theo}\label{l:Sobolev}
Assume $\mG$ satisfies the \emph{$d$-isoperimetric inequality} \eqref{e:isoperimetric constant} for some $d\ge 2$.
If $\mG$ is uniformly locally finite, then there exists a constant $C=C(C_1,d,C_d)$  such that for any $p\in [1,d)$ and $f\in c_{00}(\mV)$
\begin{equation*}
\|f\|_{\ell^{p^*}_\nu(\mV)}\leq C\frac{p}{d-p}\|\mathcal I^T f\|_{\ell^p_\mu(\mE)}\ ,
\end{equation*}
where
\begin{equation}\label{eq:pstar}
p^*:=\frac{dp}{d-p}>p\
\end{equation}
and
\begin{equation*}C_1:=\inf_{\mv\in \mV} \frac{\nu(\mv)}{\mu(\mv)}>0\ .
\end{equation*}
\end{theo}
\begin{proof} For the case $p=1$, it follows from~\cite[Prop.~4.3]{Woe00} that
\[
\|f\|_{\ell^{1^*}_{\nu}(\mV)}\leq C_d\|\mathcal I^T f\|_{\ell^1_\mu(\mE)}\qquad\hbox{for all }f\in c_{00}(\mV)\ .
\]
For the proof of the theorem, we may without loss of generality assume that $f\geq 0,$ since for each $f\in \mathbb R^\mV$ and each $\me$
\[|\ \mathcal I^T|f|(\me)\ |\le |\mathcal I^Tf(\me)|.
\] For $p>1,$
set $\alpha:=\frac{p^*}{1^*}=\frac{p^*}{q}+1>1$ where $q$ is the H\"older conjugate of $p,$ i.e.\ $1/p+1/q=1.$ Note that for $p\in (1,d),$ $q\in (\frac{d}{d-1},\infty).$ For any $\me\in \mE,$ by the mean value theorem, there exists a number $\xi$ between $f(\me_{-})$ and $f(\me_{+})$ such that \begin{eqnarray*}|f^{\alpha}(\me_{+})-f^{\alpha}(\me_{-})|&=&\alpha \xi^{\alpha-1}|f(\me_{+})-f(\me_{-})|\\
&\leq& \alpha(f(\me_{-})\vee  f(\me_{+}))^{\alpha-1}|f(\me_{+})-f(\me_{-})|,\end{eqnarray*} where $f(\me_{-})\vee f(\me_{+}):=\max\{f(\me_{-}),f(\me_{+})\}.$
Then
\[
\begin{split}\|f^{\alpha}\|_{\ell^{1^*}_\nu(\mV)}&\leq C_d\|\mathcal I^T f^{\alpha}\|_{\ell^1_\mu(\mE)}\\
&=C_d\sum_{\me\in \mE}|f^{\alpha}(\me_{+})-f^{\alpha}(\me_{-})|\mu(\me)\\
&\leq C_d\alpha\sum_{\me\in \mE}(f(\me_{+})\vee f(\me_{-}))^{\alpha-1}|f(\me_{+})-f(\me_{-})|\mu(\me)\\
&\leq C_d\alpha \left[\sum_{\me\in \mE} \left(f(\me_{+})\vee f(\me_{-})\right)^{(\alpha-1)q}\mu(\me)\right]^{1/q}\left(\sum_{\me\in \mE}\left|f(\me_{+})-f(\me_{-})\right|^p\mu(\me)\right)^{1/p}\\
&\leq C_d\alpha \left(\sum_{\mv \in \mV} f^{(\alpha-1)q}(\mv)\mu(\mv)\right)^{1/q} \|\mathcal I^T f\|_{\ell^p_\mu(\mE)}
\\
&\leq C_d\alpha C_1^{-1/q} \left(\sum_{\mv \in \mV} f^{(\alpha-1)q}(\mv)\nu(\mv)\right)^{1/q} \|\mathcal I^T f\|_{\ell^p_\mu(\mE)}\\
&\leq C_d\alpha (C_1^{-\frac{d-1}{d}}\vee1)\left(\sum_{\mv \in \mV} f^{(\alpha-1)q}(\mv)\nu(\mv)\right)^{1/q} \|\mathcal I^T f\|_{\ell^p_\mu(\mE)}.
\end{split}
\]
This yields that $\|f\|_{\ell^{p^*}_\nu(\mV)}\leq  C\frac{p}{d-p}\|\mathcal I^T f\|_{\ell^p_\mu(\mE)}.$
\end{proof}

\begin{rem}
There are large families of graphs satisfying $d$-isoperimetric inequalities: E.g., any Cayley graph of a finitely generated (discrete) group with volume growth faster than $r^d$, i.e.,
\[
\mathrm{vol}(B_r):=\# \{\mv\in \mV: {\rm dist \ }(\mv,e)\le r \}\geq Cr^d\ ,\qquad r\in \mathbb R_+\ ,
\] where $e$ is the identity element of the group, satisfies a $d$-isoperimetric inequality \eqref{e:isoperimetric constant}: a canonical example is given by the additive group $\mathbb Z^d$. Furthermore, this condition is invariant under rough isometry: More precisely, if two uniformly locally finite graphs with bounded geometry (w.r.t. both $\mu$ and $\nu$) are roughly isometric and one of them satisfies \eqref{e:isoperimetric constant}, then so does the other, see~\cite[Thm.~4.7]{Woe00}. This indicates that isoperimetric inequalities impose on graphs global conditions that are stable under local perturbations.
\end{rem}

The following lemma is the key to prove the finite extinction time property, i.e.\ Theorem~\ref{theo:finiteproof}.
\begin{lemma}\label{lem:fol}
 For any $1\leq p<\frac{2d}{d+1}$ and $d>1,$ set
\[
m:=d\left(\frac2p-1\right)>1\quad\hbox{and}\quad s:=\frac{m+p-2}{p}\ .
\]
Then  for any $a>b>0,$
$$\frac{a^{m-1}-b^{m-1}}{a-b}\geq \frac{m-1}{s^p}\left(\frac{a^s-b^s}{a-b}\right)^p.$$
\end{lemma}
\begin{proof}
By the Newton--Leibniz formula and the convexity of $s\mapsto s^p$ for $p\geq 1,$ one has
\begin{eqnarray*} \left(\frac{a^s-b^s}{a-b}\right)^p&=&\left(\frac{s}{a-b}\int_b^{a}\xi^{s-1}d\xi\right)^p\\
&\leq& \frac{s^p}{a-b}\int_b^{a}\xi^{(s-1)p}d\xi=\frac{s^p}{a-b}\int_b^{a}\xi^{m-2}d\xi\\
&=&\frac{s^p}{m-1}\frac{a^{m-1}-b^{m-1}}{a-b},
\end{eqnarray*}
as we wanted to prove.
\end{proof}

\begin{proof}[Proof of Theorem~\ref{theo:finiteproof}.] It suffices to consider $f_0\geq 0$ since for the solution $t\mapsto e^{-t\Lfun_p} f_0$ of the parabolic $p$-Laplace equation with initial data $f_0,$ the inequality $|e^{-t\Lfun_p} f_0|\leq e^{-t\Lfun_p}|f_0|$ holds pointwise by~\cite[Prop.~3.9]{Mug13}.
Because $1<p<\frac{2d}{d+1}$ one has $m=d(\frac{2}{p}-1)>1$. Set $s:=\frac{m+p-2}{p}$. Now, the functions $(\varphi(t,\cdot)-k)_+$ are finitely supported for any $t\geq 0$ and any $k>0$, since the solution $\varphi(t,\cdot):=e^{-t\Lfun_p}f_0\in \ell^m_\nu(\mV)$ and $\inf_{\mv\in \mV}\nu(\mv)>0,$ see Remark~\ref{rem:atinfin}.

We have for each time $t$ and each node $\mv$
 $$\frac{d}{dt}(\varphi-k)_+^m=m(\varphi-k)_+^{m-1}\frac{d}{dt}(\varphi-k)_+ = m(\varphi-k)_+^{m-1}\frac{d}{dt}\varphi\ .$$
We let $g:=(\varphi-k)_+$ and using Definition~\ref{defi-thm:mainnu} we compute
\[
\begin{split}\frac{d}{dt}\sum_{\mv \in \mV}(\varphi(t,\mv)-k)_+^m\nu(\mv)&=m\sum_{\mv \in \mV} (\varphi(t,\mv)-k)_+^{m-1}\frac{d}{dt}\varphi(t,\mv)\nu(\mv)\\
&=-m\sum_{\mv \in \mV} \Lfun_p \varphi(t,\mv) (\varphi(t,\mv)-k)_+^{m-1}\nu(\mv)\\
&=-m\sum_{\me\in \mE} |\mathcal I^T \varphi(t,\me)|^{p-2}\mathcal I^T \varphi(t,\me)\mathcal I^T((\varphi(t,\cdot)-k)_+^{m-1})(\me)\mu(\me)\\
&\leq -m\sum_{\me\in  \mE} |\mathcal I^T g(\me)|^{p-2}\mathcal I^T g(\me)\mathcal I^T (g^{m-1}(\me))\mu(\me)\\
&\leq -\frac{m(m-1)}{s^p}\sum_{\mv,\mw\in \mV}|\mathcal I^T (g^s)(\me)|^p\mu(\me)\ ,
\end{split}
\]
where we use Lemma~\ref{lem:fol} in the last inequality. Hence by Theorem~\ref{l:Sobolev},
\[
\begin{split}
\frac{d}{dt}\sum_{\mv \in \mV}(\varphi(t,\mv)-k)_+^m\nu(\mv)&\leq-\overline{C}\left(\sum_{\mv \in \mV}g^{sp^*}(\mv)\nu(\mv)\right)^{\frac{p}{p^*}}\\
&=-\overline{C}\left(\sum_{\mv \in \mV}g^{m}(\mv)\nu(\mv)\right)^{\frac{p}{p^*}}=-\overline{C}\left(\sum_{\mv \in \mV}(\varphi(t,\mv)-k)_+^{m}(\mv)\nu(\mv)\right)^{\frac{p}{p^*}},
\end{split}
\] where $p^*$ is defined in~\eqref{eq:pstar} and
\begin{equation}\label{eq:cbar}
\overline{C}:=C^p\frac{m(m-1)(d-p)^p}{(m+p-2)^p}\ ,
\end{equation}
for the constant $C=C(C_1,d,C_d)$ appearing in Theorem~\ref{l:Sobolev}.

By setting
\[
E_{m,k}(t):=\sum_{\mv \in \mV} (\varphi(t,\mv)-k)_+^m\nu(\mv)
\]
we have
\begin{equation}\label{eq:pp*}
\frac{d}{dt}E_{m,k}(t)+\bar{C} E_{m,k}(t)^{\frac{p}{p^*}}\leq 0\ .
\end{equation}

Accordingly, by integration over $t,$ we have
\[\left\{\begin{array}{ll}
E_{m,k}(t)\leq \left(E_{m,k}(0)^{p/d}-\frac{\overline{C}p}{d} t\right)^{d/p},\qquad &\hbox{for }t\in \left(0, t_{0,k}\right]\ , \\
E_{m,k}(t)=0,&\hbox{for }t\in \left[t_{0,k},\infty\right)\ ,
\end{array}\right.
\]
where
\[
t_{0,k}:=\frac{d}{\overline{C}p}E_{m,k}(0)^{p/d}\ .
\]
By passing to the limit $k\to 0$ we get $\varphi(t,\cdot)\equiv 0,$ for $t\geq t_0=\lim\limits_{k\to 0}t_{0,k}.$ This proves the theorem.
\end{proof}

\begin{rem}\label{rem:finite extinction}
\begin{enumerate}[(a)]\item
In the proof of the above theorem, the extinction time can be estimated as
$$T_0=\frac{d(m+p-2)^p}{pC^pm(m-1)(d-p)^p}\|f_0\|_{\ell^{d\frac{2-p}{p}}_{\nu}}^{2-p},$$ where $C$ is a constant depending on $d,C_1,C_d$ introduced in the statement of Theorem~\ref{l:Sobolev} and $m:=d(\frac2p-1)>1$.
\item
Throughout this article we have only considered the $p$-Laplacian for $p>1$. This is due to the fact that our analysis is based on variational methods applied to the energy functional $f\mapsto \|\mathcal I^T f\|^p_{\ell^p_\mu}$, but for $p=1$ this functional is not Fréchet differentiable and in fact its subdifferential -- the $1$-Laplacian -- is a multi-valued operator.
We stress that the finite extinction time $T_0$ does not blow up as $p\to 1+$, but rather it converges to
\[
\frac{1}{C(d-1)}\|f_0\|_{\ell^d_\nu}\ .
\]
\end{enumerate}
\end{rem}

\subsection{Conservation of mass}
To the best of our knowledge, conservation of mass for the parabolic $p$-Laplace equation~\eqref{i:eq-parabolic} on graphs has not been proved in the literature to date. Our proof closely follows the proof for the corresponding equation~\eqref{i:eq-parabolic-omega} on domains in~\cite[\S~3.3.3]{Vaz07}.

\begin{proof}[Proof of Theorem~\ref{thm:conservation of mass}]
We know from Theorem~\ref{thm:main-new} and Corollary~\ref{ralphwell-local} that the solution $\varphi$ satisfies $\varphi(t,\cdot)\in \ell^1_{\nu}$ for all $t>0$ if $f_0\in \ell^1_{\nu}$. By the assumption that $\inf_{\mv\in V}\nu(\mv)>0,$ we have  $\varphi(t,\cdot)\in \ell^{p-1}_{\nu}$ since $p\geq2$.

Let $h:\mV\to \mathbb R$ be a function with finite support that attains thereon constant value 1. Then by Theorem~\ref{thm:main-new} and Definition~\ref{defi->thm:main} one has
\begin{equation}\label{eq:conservmass}
\begin{split}
\frac{d}{dt}\sum_{\mv\in \mV}\varphi(t,\mv)h(\mv)\nu(\mv)&=-\sum_{\mv\in \mV}\Lfun_p\varphi(t,\mv)h(\mv)\nu(\mv)\\
&=-\sum_{\me\in \mE} \mu(\me)|({\mathcal I^T}\varphi)(\me)|^{p-2}({\mathcal I^T}\varphi)(\me) ({\mathcal I^T}h)(\me)\ .
\end{split}
\end{equation}

Letting $h$ converge from below to a nodewise constant function of value 1, ${\mathcal I^T}h$ converges to the edgewise constant function of value 0 with a uniform upper bound $|{\mathcal I^T}h(\me)|\leq2$, $\me\in \mE$, whereas $|({\mathcal I^T}\varphi)|^{p-2}({\mathcal I^T}\varphi)\in \ell^1_{\mu}(\mE)$ follows from $\varphi\in \ell^{p-1}_{\nu}(\mV)$ and the uniformly local finiteness. By Lebesgue's dominated convergence theorem we conclude that the terms in~\eqref{eq:conservmass} tend to 0. Consequently,
\[
\sum_{\mv\in \mV}\varphi(t,\mv)\nu(\mv)=\sum_{\mv\in \mV}\varphi(0,\mv)\nu(\mv)=\sum_{\mv\in \mV}f_0(\mv)\nu(\mv),\qquad t\ge 0\ ,
\]
i.e., conservation of mass does indeed hold.
\end{proof}

In order to prove the strict positivity of the solutions in Theorem~\ref{prop:irreducibility}, we need the following lemma, which may be of independent interest. For the case of domains, the corresponding domination property is well-known, see e.g.~\cite[\S~4]{Ouh96} and~\cite[\S~5]{Bar96} for the cases of $p=2$ and $p\in (1,\infty)$, respectively.
\begin{lemma}\label{lem:domin}
The semigroup generated by $-\Lfun^{\rm D}_p$ is dominated by the semigroup generated by $-\Lfun^{\rm N}_p$, i.e.,
\[
e^{-t\Lfun_p^{\rm D}} f\le e^{-t\Lfun_p^{\rm N}} f\qquad\hbox{ for all }0\le f\in \ell^2_\nu(\mV)\hbox{ and all }t\ge 0.
\]
\end{lemma}
\begin{proof}
By~\cite[Cor.~2.9]{Bar96}, it is sufficient to check that $\mathring{w}^{1,p,2}_{\mu,\nu}(\mV)$ is a \emph{positive ideal} of $w^{1,p,2}_{\mu,\nu}(\mV)$ (cf.~\cite[\S~2]{Bar96} for details). This follows from \cite[Lemma~3.11]{Mug14}.
\end{proof}

\begin{proof}[Proof of Theorem~\ref{prop:irreducibility}.] Let $\mG$ be a finite graph (without any boundary condition) and $\phi$ be the solution to the parabolic $p$-Laplace equation with nonnegative and nontrivial initial data $f_0$. Then by~\cite[Prop.~3.9]{Mug13} $\varphi(t,\mv)\geq 0$ for all $t\ge 0$ and all $\mv\in \mV$. In order to show the claim by contradiction, let us assume that $\varphi(t_0,\mv_0)= 0$ for some $t_0>0$ and $\mv_0\in \mV$. Hence
\[
\begin{split}
0&\ge \frac{d \varphi}{dt}(t_0,\mv_0)=-\Lfun_p \varphi(t_0,\mv_0)\\
&=\frac{1}{\nu(\mv)}\sum_{\substack{\mw\in\mV \\ \mw\sim\mv_0}}\mu_{\mv_0\mw}|\varphi(t_0,\mw)-\varphi(t_0,\mv_0)|^{p-2}(\varphi(t_0,\mw)-\varphi(t_0,\mv_0))\\
&=\frac{1}{\nu(\mv)}\sum_{\substack{\mw\in\mV \\ \mw\sim\mv_0}}\mu_{\mv_0\mw}|\varphi(t_0,\mw)|^{p-2}\varphi(t_0,\mw) \ .
\end{split}
\]
Since $\varphi\ge 0$, it follows that $\varphi(t_0,\mw)=0$ for all nodes $\mw$ that are adjacent to $\mv_0$. We thus find by recursion (remember that $\mG$ is always assumed to be connected!) that $\phi(t_0,\mv)=0$ for all $\mv\in \mV$, a contradiction to the conservation of mass property stated in Theorem~\ref{thm:conservation of mass}.

Let now $\mG$ be an infinite graph. By Theorem~\ref{thm:main-2} the solution to $\rm(HE^{\rm D}_p)$ can be obtained by means of a Galerkin scheme. For any finite subgraph $\Omega\subset \mV$ with Dirichlet boundary condition, the positivity of the solutions has been proved in~\cite[Thm.\ 4.3]{LeeChu12}. Applying this result to each finite-dimensional approximation $\varphi_n,$  we have $\varphi_n(t,\mv)>0$ whenever $\mv$ lies in the support of $\varphi_n(0,\cdot).$ Again by~\cite[Prop.~3.9]{Mug13} we deduce that for any $\mv\in \mV$
\[
f(t,\mv)=e^{-t\Lfun_p^{\rm D}}f_0(\mv)=\lim_{n\to \infty}e^{-t\Lfun_p^{(n)}}f_{0_n}(\mv)>0\ .
\]
This concludes the proof.
\end{proof}

\subsection{Finite measure graphs}

We can finally prove Theorem~\ref{prop:polynomial} as follows, inspired by~\cite[Thm.~11.9]{Vaz07} and~\cite[Prop.~4.6]{MugNic13}.

\begin{proof}[Proof of Theorem~\ref{prop:polynomial}]
Take a solution $\varphi$ of $\rm(HE^{\rm N}_p)$ as above, denote $\overline{\varphi}=\frac{1}{\nu(V)}\sum_{x\in V}\varphi(t,\mv)\nu(\mv)$. By Theorem~\ref{thm:conservation of mass}, $\overline{\varphi}=\overline{f_0}$ for any $t>0.$ Hence
\[
\begin{split}
\frac12 \frac{d\|\varphi(t,\cdot)-\overline{\varphi}\|_{\ell^2_\nu}^2}{dt}&=\left( \frac{d\varphi}{dt}(t,\cdot)-\frac{d}{dt}\overline{\varphi}\mid \varphi(t,\cdot)-\overline{\varphi}\right)_{\ell^2_\nu}\\
&=-\left( \Lfun_p \varphi(t,\cdot)\mid \varphi(t,\cdot)-\overline{\varphi}\right)_{\ell^2_\nu}\\
&=-\left( |\mathcal I^T \varphi|^{p-2}(t,\cdot)\mathcal I^T \varphi(t,\cdot)\mid \mathcal I^T \varphi(t,\cdot)\right)_{\ell^2_\mu(\mE)}\\
&=-\|\mathcal I^T \varphi(t,\cdot)\|^p_{\ell^p_\mu(\mE)}\\
&\leq -C\|\mathcal I^T \varphi(t,\cdot)\|_{\ell^2_{\mu}(\mE)}^p,
\end{split}
\] in the last step we use the fact that $G$ has finite total $\mu$-measure, i.e.\ $\mu(\mV)<\infty,$ which follows from \eqref{a:ULF} and $\nu(\mV)<\infty.$

By the Poincaré inequality $\rm(PI)$, $\|\mathcal I^T \varphi(t,\cdot)\|_{\ell^2_{\mu}(\mE)}\geq C\|\varphi-\overline{\varphi}\|_{\ell^2_{\nu}}.$
 Let $\alpha:=\frac{p}{2}>1$ and
\[
\psi(t):=\|\varphi(t,\cdot)-\overline{\varphi}\|^2_{\ell^2_\nu(\mV)}\ ,
\]
so that
\begin{equation}\label{mn13-b}
\psi'(t)\le -C_1 \psi(t)^\alpha,\qquad t\ge 0\ ,
\end{equation}
 and consider the function $\varphi:[0,\infty)\to \mathbb R$ defined by
$$
\varphi(t):=C_1 (\alpha-1) t-\psi(t)^{1-\alpha},\qquad t\ge 0\ .
$$
(In order to avoid trivialities we can assume that $\psi(t)>0$ for all $t>0$.)
A direct computation shows that
$$
\varphi'(t)=(\alpha-1) \left(C_1+\psi^{-\alpha}(t) \psi'(t)\right),\qquad t> 0\ ,
$$
hence we deduce by \eqref{mn13-b} that $\varphi$ is nonincreasing, i.e.,
$$
\varphi(t)\leq \varphi(0)=-\left(\|f_0-\overline{f_0}\|^{p}_{\ell^{2}_\nu(\mV)}\right)^{1-\alpha}\leq 0,\quad t>0\ .
$$
This yields that
$$
\|\varphi(t,\cdot)-\overline{\varphi}\|^{2}_{\ell^2_\nu(\mV)}\leq C t^{-\frac{1}{\alpha-1}},\quad t>0\ ,
$$
for
\[
C:=\left(C_1(\alpha-1)\right)^{-\frac{1}{\alpha-1}}\ .
\]
This concludes the proof.
\end{proof}

By the way, observe that a uniformly locally finite graph with finite measure cannot satisfy a $d$-isoperimetric inequality for any $d>1$, since in this case the volume growth $\mathrm{\nu}(B_r(x))\geq Cr^d$ would then hold.

\end{document}